\documentclass{article}
\usepackage[utf8]{inputenc}
\pdfoutput=1
\usepackage[sc,osf]{mathpazo}
\usepackage[height=8.5in,width=5.9in,letterpaper]{geometry}
\usepackage[sort&compress,numbers]{natbib}
\usepackage[colorlinks=true,citecolor=blue,breaklinks]{hyperref}
\usepackage[hyphenbreaks]{breakurl} 
\usepackage[tracking]{microtype}

\makeatletter
\newlength\aftertitskip     \newlength\beforetitskip
\newlength\interauthorskip  \newlength\aftermaketitskip

\def\maketitle{\par
 \begingroup
   \def\thefootnote{\fnsymbol{footnote}}
   \def\@makefnmark{\hbox to 4pt{$^{\@thefnmark}$\hss}}
   \@maketitle \@thanks
 \endgroup
\setcounter{footnote}{0}
 \let\maketitle\relax \let\@maketitle\relax
 \gdef\@thanks{}\gdef\@author{}\gdef\@title{}\let\thanks\relax}

\def\@startauthor{\noindent \normalsize\bf}
\def\@endauthor{}
\def\@starteditor{\noindent \small {\bf Editor:~}}
\def\@endeditor{\normalsize}
\def\@maketitle{\vbox{\hsize\textwidth
 \linewidth\hsize \vskip \beforetitskip
 {\begin{center} \LARGE\@title \par \end{center}} \vskip \aftertitskip
 {\def\and{\unskip\enspace{\rm and}\enspace}%
  \def\addr{\small\it}%
  \def\email{\hfill\small\tt}%
  \def\name{\normalsize\bf}%
  \def\AND{\@endauthor\rm\hss \vskip \interauthorskip \@startauthor}
  \@startauthor \@author \@endauthor}
}}

\makeatother

\pdfoutput=1                    

\usepackage{amsmath,bm}
\usepackage{stmaryrd}
\usepackage{xspace}
\usepackage{amsthm}
\usepackage{paralist}
\usepackage{graphicx}
\usepackage[numbers]{natbib}
\usepackage{mathdots}
\usepackage{graphicx}

\newtheorem{theorem}{Theorem}[section]

\newtheorem{prop}[theorem]{Proposition}
\newtheorem{corr}[theorem]{Corollary}

\theoremstyle{definition}

\numberwithin{equation}{section}

\newcommand{\nlsum}{\sum\nolimits}
\newcommand{\nlprod}{\prod\nolimits}
\newcommand{\R}{\mathbb{R}}

\newcommand{\pfrac}[2]{\left(\tfrac{#1}{#2}\right)}
\newcommand{\half}{\tfrac{1}{2}}

\newcommand{\frob}[1]{\|{#1}\|_{\text{F}}}

\newcommand{\prece}{\prec_E}
\newcommand{\preceeq}{\preceq_E}

\newcommand{\da}{\downarrow}

\DeclareMathOperator{\trace}{tr}

\begin{document}
\title{Logarithmic inequalities under an elementary symmetric polynomial dominance order}
\author{Suvrit Sra \email{suvrit@mit.edu}\\ \addr{Laboratory for Information and Decision Systems (LIDS), MIT,  Cambridge, MA 02139}}

\maketitle


\begin{abstract}
  We consider a dominance order on positive vectors induced by the elementary symmetric polynomials. Under this dominance order we provide conditions that yield simple proofs of several monotonicity questions. Notably, our approach yields a quick (4 line) proof of the so-called \emph{``sum-of-squared-logarithms''} inequality conjectured in (P.~Neff, B.~Eidel, F.~Osterbrink, and R.~Martin, \emph{Applied Math. \& Mechanics., 2013}; P.~Neff, Y.~Nakatsukasa, and A.~Fischle; \emph{SIMAX, 35, 2014}). This inequality has been the subject of several recent articles, and only recently it received a full proof, albeit via a more elaborate complex-analytic approach. We provide an elementary proof, which moreover extends to yield simple proofs of both old and new inequalities for R\'enyi entropy, subentropy, and quantum R\'enyi entropy.
\end{abstract}

\maketitle

\section{Introduction}
Let $x$ be a real vector with $n$ components. Let $e_k$ denote the $k$-th elementary symmetric polynomial defined by
\begin{equation*}
  e_k(x) := \sum_{1 \le i_1 < \cdots < i_k \le n} \prod_{j=1}^k x_{i_j}.
\end{equation*}
For nonnegative vectors  $x, y$ in $\R_+^n$, we consider the dominance order $\prece$ induced by the elementary symmetric polynomials. More precisely, we say $x \prece y$ if 
\begin{equation}
  \label{eq:2}
  e_k(x) \le e_k(y),\quad k=1,\ldots,n-1,\quad\text{and}\quad e_n(x)=e_n(y).
\end{equation}
If the last equality is just an inequality $e_n(x) \le e_n(y)$, we write $x \preceeq y$. We consider functions that are monotonic under the partial order $\prec_E$. Specifically, we say a function $F: \R_+^n \to \R$ is \emph{E-monotone} if
\begin{equation}
  \label{eq:3}
  x \prece y\quad\implies\quad F(x) \le F(y).
\end{equation}

This paper is motivated by a body of recent papers that study E-monotonicity of a specific function: the so-called \emph{``sum-of-squared-logarithms''} $L_n(x) = \sum_{i=1}^n (\log x_i)^2$. Indeed, $L_n(x)$ has been the focus of several works~\cite{neff13,neff2012,PompeNeff2015,yuji.ssli}, wherein the key open question was  establishing its E-monotonicity. The works~\cite{neff2012,neff13,PompeNeff2015} establish E-monotonicity for $n=2,3,4$;
The authors of~\cite{yuji.ssli} also highlighted the powerful implications of the general case towards solving certain nonconvex optimization problems to global optimality. Only very recently, a full solution was obtained via a complex analysis~\cite{borisov2017,ssli.mo}. While preparing this paper, it was brought to our notice~\cite{neff.private} that \cite{silhavy15} has obtained a characterization of E-monotone functions via the theory of Pick functions.\footnote{E-monotonicity of $L_n$ has additional interesting history. P.~Neff offered a reward of one ounce of fine gold for its proof, a conjecture that he also announced on the MathOverflow platform~\cite{ssli.mo}. Shortly thereafter, the first full proof was sketched by L.~Borisov using contour integration~\cite{ssli.mo}. Approximately two weeks after Borisov's proof, Šilhavý independently characterized E-monotone functions~\cite{neff.private}. His results are based on the theory of Pick functions, a natural and elegant approach to study E-monotonicity, which was in foreshadowed in the remarkable work of Josza and Mitchison~\cite{jozsa2015}.} Our work offers a complementary, and in our view, perhaps the simplest perspective, which yields a short (4 line) proof of E-monotonicity of $L_n$ as a byproduct. 



\section{E-monotonicity}
We introduce now our elementary approach, which leads to a short proof of the E-monotonicity of $L_n$ as well as similar results for related entropy and sub-entropy inequalities of~\cite{jozsa2015}. Our proof technique should generalize to monotonicity induced by other symmetric polynomials (e.g., Schur polynomials); we leave such an exploration to the interested reader.

Our main result is the following simple, albeit powerful sufficient condition:
\begin{prop}
  \label{prop:psi}
  Let $\psi$ be real-valued function admitting the representation $$\psi(s)=\int_0^a\log(t+s)d\mu(t),\quad\text{or}\quad \psi(s) = \int_0^a\log(1+t s)d\mu(t),$$ where $a > 0$, $s \ge 0$, and $\mu$ is nonnegative measure. Then, $\sum_{i=1}^n \psi(x_i)$ is E-monotone. 
\end{prop}
\begin{proof}
  Recall first the generating functions for elementary symmetric polynomials
  \begin{align*}
    \nlsum_{k=0}^nt^ke_{k}(x) &=\nlprod_{i=1}^n(1+ t x_i),\\
    \nlsum_{k=0}^nt^ke_{n-k}(x) &= \nlprod_{i=1}^n(t+x_i).
  \end{align*}
  Let $x, y \in \R_+^n$, and suppose $x\preceeq y$. Then using the above generating function representation under this hypothesis we immediately obtain
  \begin{align}
    \label{eq:key1}
    \nlprod_{i=1}^n(1+ t x_i) &\le\nlprod_{i=1}^n(1+ty_i)\qquad\forall t \ge 0\\
    \label{eq:key2}
    \nlprod_{i=1}^n(t+x_i) &\  \le \nlprod_{i=1}^n(t+y_i)\qquad\forall t \ge 0.
  \end{align}
  Taking logarithms, multiplying by $d\mu(t)$, and integrating, it then follows that
  \begin{equation*}
    \begin{split}
      &\sum_{i=1}^n\int_0^a \log(1+tx_i)d\mu(t)
      \le 
      \sum_{i=1}^n\int_0^a \log(1+ty_i)d\mu(t),\\
      &\implies\ F(x)=\nlsum_i \psi(x_i) \le \nlsum_i \psi(y_i)=F(y).
    \end{split}
  \end{equation*}
  Similarly, with~\eqref{eq:key2} we again obtain $F(x)=\sum_i \psi(x_i) \le \sum_i \psi(y_i)=F(y)$.
\end{proof}

\noindent\textbf{Remark.} Observe that the $E$-monotonicity relation is \emph{weaker} than the usual majorization order. Indeed, if $x \prec y$ (i.e., $\sum_{i=1}^k x_i^\da \le \sum_{i=1}^k y_i^\da$ for $1\le k < n$, and $x^T1=y^T1$), then $e_k(x) \ge e_k(y)$ because $e_k$ is Schur-concave~\cite{marOlk}. 

\subsection{Proof of the SSLI}
As an immediate corollary to Prop.~\ref{prop:psi} we obtain the announced E-monotonicity of $L_n(x)=\sum_{i=1}^n(\log x_i)^2$.
\begin{corr}
  \label{cor:ssli}
  Let $x, y \in \R_+^n$ such that $x \prece y$. Then, $L_n(x) \le L_n(y)$.
\end{corr}
\begin{proof}
  The key is to rewrite $(\log x)^2$ so that Prop.~\ref{prop:psi} applies. We observe that
  \begin{equation}
    \label{eq:10}
    (\log x)^2 = \int_0^\infty\log \left(\frac{(1 + tx) \left(t+x\right)}{x(1+t)^2}\right)\frac{dt}{t}.
  \end{equation}
  Next, using inequalities~\eqref{eq:key1} and \eqref{eq:key2}, and the assumption $e_n(x)=e_n(y)$ (whereby $\sum_i\log(rx_i)=\sum_i \log(ry_i)$ for $r>0$) we obtain the inequality
  \begin{equation*}
      \nlsum_i\log(1+tx_i)(t+x_i) - \log\bigl((1+t)^2x_i\bigr)
      \le
      \nlsum_i\log(1+ty_i)(t+y_i) - \log\bigl((1+t)^2y_i\bigr).
  \end{equation*}
  Integrating this over $t$ with $d\mu(t)=\frac{dt}{t}$ and using identity~\eqref{eq:10} the proof follows. 
\end{proof}

\subsection{Entropy}
Now we consider application of Prop.~\ref{prop:psi} to obtain entropy inequalities.
Recall that for a probability vector $x$, the R\'enyi entropy of order $\alpha$, where $\alpha\ge0$ and $\alpha\neq 1$, is defined as
\begin{equation}
  \label{eq:15}
  H_\alpha(x) := \frac{1}{1-\alpha}\log\left(\nlsum_{i=1}^n x_i^\alpha\right).
\end{equation}
The limiting value $\lim_{\alpha \to 1}H_\alpha$ yields the usual (Shannon) entropy $-\sum_ix_i\log x_i$.
\begin{theorem}
  \label{thm:renyi}
  Suppose $x$ and $y$ probability vectors. Then, 
  \begin{equation*}
    x \prece y \implies H_\alpha(x) \le H_\alpha(y)\quad\text{for}\ \alpha \in [0,2].
  \end{equation*}
\end{theorem}
\begin{proof}
  Since $\log$ is monotonic, to analyze E-monotonicity of $H_\alpha$, it suffices to consider the following three special cases:
  \begin{subequations}
    \begin{align}
      \label{eq:4}
      \nlsum_{i=1}^n x_i^\alpha &\le \nlsum_{i=1}^n y_i^\alpha,\qquad \text{if}\ 0 < \alpha < 1,\quad\text{and}\quad e_1(x)=e_1(y),\\
      \label{eq:5}
      \nlsum_{i=1}^n x_i^\alpha &\ge \nlsum_{i=1}^n y_i^\alpha,\qquad \text{if}\ 1 < \alpha < 2,\quad\text{and}\ e_1(x)=e_1(y),\\
      \label{eq:6}
      -\nlsum_{i=1}^n x_i\log x_i &\le -\nlsum_{i=1}^n y_i\log y_i,\quad\text{and}\quad e_1(x)=e_1(y).
    \end{align}
  \end{subequations}
  Observe that for $0 < \alpha < 1$ and $s\ge0$, we have the integral representation
  \begin{equation}
    \label{eq:7}
    s^\alpha = \frac{\alpha\sin(\alpha\pi)}{\pi}\int_{0}^\infty \log(1+ts)t^{-\alpha-1}dt.
  \end{equation}
  Given~\eqref{eq:7}, an application of Prop.~\ref{prop:psi} immediately yields~\eqref{eq:4}. 

  For~\eqref{eq:5}, we consider a different representation (notice the extra $ts$ term):
  \begin{equation}
    \label{eq:8}
    s^\alpha = \frac{\alpha\sin(\alpha\pi)}{\pi}\int_0^\infty\left(\log(1+ts)-ts\right)t^{-\alpha-1}dt.
    \end{equation}
    This integral converges for $1 < \alpha < 2$ and $s \ge 0$. Since $x\prece y$ and we assumed $e_1(x)=e_1(y)$, it follows that $\nlsum_i \bigl(\log(1+tx_i) - tx_i\bigr)   \le   \nlsum_i \bigl(\log(1+ty_i) - ty_i\bigr)$. Thus, using~\eqref{eq:8} and noting that $\sin(\alpha\pi) < 0$ for $1<\alpha<2$, we obtain~\eqref{eq:5}. 

    To obtain~\eqref{eq:6} we apply a limiting argument to~\eqref{eq:5}. In particular, recall that
    \begin{equation*}
      \lim_{\alpha\to 1}\frac{x_i^\alpha-x_i}{\alpha-1}=x_i\log x_i,
    \end{equation*}
    so that upon using $\sum_ix_i=\sum_iy_i$ in~\eqref{eq:5}, dividing by $\alpha-1$, and taking limits as $\alpha\to 1$, we obtain~\eqref{eq:6}.
\end{proof}



\subsection{Inequalities for positive definite matrices}
\label{sec:mtx}
We note below some inequalities on (Hermitian) positive definite matrices that follow from the above discussion. We write $A>0$ to indicate that $A$ is positive definite. We extend the definition~(\ref{eq:2}) to such matrices in the usual way. In particular, let $A, B > 0$. We say
\begin{equation}
  \label{eq:16}
  A \prec_E B\quad\Longleftrightarrow\quad \lambda(A) \prec_E \lambda(B),
\end{equation}
where $\lambda(\cdot)$ denotes the vector of eigenvalues. Recalling that $e_k(\lambda(A)) = \trace(\wedge^k A)$, where $\wedge$ is the exterior product~\cite[Ch.~1]{bhatia97}, we obtain the following result.
\begin{prop}
  \label{prop:trivial}
  Let $A, B$ be positive definite matrices. Then,
  \begin{equation*}
    \trace(\wedge^k A) \le \trace(\wedge^kB),\quad\text{for}\ k=1,\ldots,n\  \implies\log\det(I+A) \le \log\det(I+B).
  \end{equation*}  
\end{prop}
\noindent\textbf{Remark.} A classic result in eigenvalue majorization states that if $\log\lambda(A) \prec \log\lambda(B)$ (the usual dominance order), then we have $\log\det(I+A) \le \log\det(I+B)$. Prop.~\ref{prop:trivial} presents an \emph{alternative} condition that implies the same determinantal inequality.

\vskip5pt
Let us now state two other notable consequences of the order~\eqref{eq:16}. To that end, we recall the Riemannian distance on the manifold of positive definite matrices (see e.g.,~\cite[Ch.~6]{bhatia07}) as well as the S-Divergence~\cite{ssdiv}:
\begin{align}
  \label{eq:riem}
  \delta_R(A,B) &:= \frob{\log B^{-1/2}AB^{-1/2}},\\
  \label{eq:sdiv}
  \delta_S(A,B) &:= \log\det\pfrac{A+B}{2}-\half\log\det(AB).
\end{align}
\begin{prop}
  \label{prop:riem}
  If $A, B, C > 0$ and $AC^{-1} \prec_E BC^{-1}$, then
  \begin{align}
    \label{eq:11}
    \delta_R(A,C) \le \delta_R(B,C),\\
    \label{eq:12}
    \delta_S(A,C) \le \delta_S(B,C).
  \end{align}
\end{prop}
\begin{proof}
  Inequality~\eqref{eq:11} (for $C=I$) was first noted in~\cite{boNes15,borisov2017}. It follows readily from Corollary~\ref{cor:ssli} once we use~\eqref{eq:riem} and observe that
  \begin{equation*}
    \delta_R^2(A,C) = \frob{\log C^{-1/2}AC^{-1/2}}^2 = \nlsum_{i=1}^n(\log \lambda_i(AC^{-1}))^2.
  \end{equation*}
  To obtain~\eqref{eq:12}, first observe that $$\det(A+C)=\det(C)\det(I+AC^{-1})=\det(C)\nlprod_{i=1}^n(1+\lambda_i(AC^{-1})).$$ Thus, we have
  \begin{align*}
    \delta_S(A,C) &= \log\det(C) + \log\nlprod_{i=1}^n\tfrac{1+\lambda_i(AC^{-1})}{2}-\half\log\det(AC)\\
    &\le \log\det(C) + \log\nlprod_{i=1}^n\tfrac{1+\lambda_i(BC^{-1})}{2}-\half\log\det(AC)\\
    &=\log\det\pfrac{B+C}{2}-\half\log\det(BC)\\
    &=\delta_S(B,C),
  \end{align*}
  where the inequality  holds due to the hypothesis $\lambda(A) \prece \lambda(B)$, which also is used to conclude the second equality by using $\det(A)=\det(B)$.
\end{proof}

\subsection{Quantum Entropy}
The entropy inequalities~(\ref{eq:4})-\eqref{eq:6} also extend to their counterparts in quantum information theory. Specifically, recall that the quantum R\'enyi entropy of order $\alpha \in (0,1)\cup (1,\infty)$ is given by
\begin{equation}
  \label{eq:19}
  H_\alpha(X) := \frac{1}{1-\alpha}\log \frac{\trace X^\alpha}{\trace X},
\end{equation}
where $X$ is positive definite; moreover, one typically assumes the normalization $\trace X = 1$. Using an argument of the same form as used to prove Theorem~\ref{thm:renyi} we can obtain the following result for the R\'enyi entropy; we omit the details for brevity.
\begin{theorem}
  \label{thm:renyi2}
  Let $X$ and $Y$ be positive definite matrices with unit trace. Then,
  \begin{equation*}
    X \prec_E Y\quad\implies H_\alpha(X) \le H_\alpha(Y)\quad\text{for}\ \alpha \in [0,2].
  \end{equation*}
\end{theorem}

\section{Subentropy} 
Next, we briefly discuss an important extension, namely, E-monotonicity of \emph{subentropy}, a quantity that has found use in physics~\cite{jozsa2015}. Formally,
\begin{equation}
  \label{eq:13}
  Q(x_1,\ldots,x_n) := -\sum_{i=1}^n\frac{x_i^n}{\nlprod_{j\neq i}(x_i-x_j)}\log x_i,
\end{equation}
defines a natural entropy-like quantity that characterizes a quantum state with eigenvalues $x_1,\ldots,x_n$ (thus $x \ge 0$ and $e_1(x)=1$). A main result in the work~\cite{jozsa2015} is the following monotonicity theorem for subentropy (rephrased in our notation):
\begin{theorem}[\cite{jozsa2015}]
  \label{thm:joz}
  If $x \preceeq y$ and $e_1(x)=e_1(y)=1$, then $Q(x) \le Q(y)$.   
\end{theorem}

\vskip5pt
\noindent Josza and Mitchison~\cite{jozsa2015} prove Theorem~\ref{thm:joz} by  appealing to an argument based on contour integration. We note below how a key identity derived by Josza and Mitchison already implies this theorem. Instead of the logarithmic representation of Prop.~\ref{prop:psi}, the key idea is to consider the representation
\begin{equation}
  \label{eq:1}
  \psi(x_1,\ldots,x_n) = \int_0^\infty h\left(\nlprod_{i=1}^n (t+x_i)\right)d\mu(t),
\end{equation}
where $h$ is any monotonically increasing function and $\mu$ is a nonnegative measure. Clearly, if $x \preceeq y$, then $h(\prod_i (t+x_i)) \le h(\prod_i(t+y_i))$, whereby $\psi(x) \le \psi(y)$. 

Therefore, to prove $Q(x)\le Q(y)$, we just need to find a function $h$ such that $Q$ can be expressed as~\eqref{eq:1}. Such a representation was already obtained in~\cite{jozsa2015}, wherein it is shown that for $x>0$ such that $e_1(x)=1$, we have
\begin{equation}
  \label{eq:14}
  Q(x_1,\ldots,x_n) = -\int_0^\infty \Bigl[\frac{t^n}{\nlprod_{j=1}^n(t+x_j)} - \frac{t}{1+t}\Bigr]dt.
\end{equation}
Thus, using $h(s)=-1/s$ and $d\mu(t)=t^ndt$, and adding $\frac{-t}{1+t}$ to ensure convergence (the constraint $e_1(x)=e_1(y)$ is needed to cancel out the effect of this term), we obtain $Q(x)\le Q(y)$ whenever $x \preceeq y$ and $e_1(x)=e_1(y)$.

A similar argument yields the following inequality, which is otherwise not obvious:
\begin{alignat}{2}
  \label{eq:9}
  x \prece y \implies \sum_{i=1}^n \frac{(-1)^{i+1}x_i^\alpha}{\nlprod_{j\neq i}(x_i-x_j)}
  &\quad\ge\quad
  \sum_{i=1}^n \frac{(-1)^{i+1}y_i^\alpha}{\nlprod_{j\neq i}(y_i-y_j)},\quad\text{for}\ 0 < \alpha < 1.
\end{alignat}

\bibliographystyle{abbrvnat}
\setlength{\bibsep}{3pt}

\end{document}